\documentclass[12pt]{article}
\usepackage{amsfonts}
\usepackage{amssymb}
\usepackage{mathrsfs}
\usepackage{amsmath}
\usepackage{amsthm}
\usepackage{amstext}
\usepackage{amsopn}
\usepackage{graphicx}
\usepackage{color}
\newtheorem{theorem}{Theorem}[section]
\newtheorem{lemma}[theorem]{Lemma}

\theoremstyle{definition}

\theoremstyle{remark}

\numberwithin{equation}{section}

\newcommand{\ba}{\begin{array}}
\newcommand{\ea}{\end{array}}
\newcommand{\f}{\frac}

\newcommand{\Om}{\Omega}

\newcommand{\la}{\lambda}

\newcommand{\ds}{\displaystyle}

\begin{document}
\date{}
\title{ \bf\large{Hopf bifurcation in a delayed reaction-diffusion-advection population model}\footnote{This research is supported by the National Natural Science Foundation of China (Nos. 11301111, 11571363, 11571364 and 11371111) and NSF grant DMS-1411476.}}
 \author{Shanshan Chen\textsuperscript{1}\footnote{Email: chenss@hit.edu.cn}, \ Yuan Lou\textsuperscript{2}\footnote{Email: lou@math.ohio-state.edu}, \  Junjie Wei\textsuperscript{1}\footnote{Corresponding Author, Email: weijj@hit.edu.cn} \
 \\
 {\small \textsuperscript{1} Department of Mathematics,  Harbin Institute of Technology,\hfill{\ }}\\
\ \ {\small Weihai, Shandong, 264209, P.R.China\hfill{\ }}\\
{\small \textsuperscript{2} Department of Mathematics, Ohio State University,\hfill{\ }}\\
\ \ {\small Columbus, OH 43210, USA\hfill{\ }}}
\maketitle

\begin{abstract}
{In this paper, we investigate a reaction-diffusion-advection model with time delay effect.
The stability/instability of the spatially nonhomogeneous positive steady state and the associated Hopf bifurcation are investigated when the given parameter of the model is near the principle eigenvalue of an elliptic operator. Our result implies that time delay can make the spatially nonhomogeneous positive steady state unstable for a reaction-diffusion-advection model, and the model can exhibit oscillatory
pattern through Hopf bifurcation.
}

\end{abstract}

\newpage
\section {Introduction}

During the past thirty years, delay induced instability has been investigated extensively for homogeneous reaction-diffusion equations with delay effect, and the spatial homogeneous and nonhomogeneous periodic solutions can occur through Hopf bifurcation.
For models with the homogeneous Neumann boundary conditions, researchers were mainly concerned with the Hopf bifurcation near the constant positive equilibrium, see \cite{ChenSW,faria2000,Faria2001,gourley2002dynamics,gourley2004nonlocality,Hadeler,Hu,Lee,Sen,Shi} and the references therein. For models with the homogeneous Dirichlet boundary conditions, the positive equilibrium is always spatially nonhomogeneous.
Busenberg
and Huang \cite{busenberg1996stability} first studied the Hopf bifurcation near such spatially nonhomogeneous positive equilibrium, and they found that, for the following prototypical single population model,
\begin{equation}\label{pp}
\begin{cases}
\ds\frac{\partial u(x,t)}{\partial t} =d\Delta u(x,t)+\la
u(x,t)\left(1-u(x,t-\tau)\right),&x\in\Om,\ t>0,\\
u(x,t)=0,&  x\in\partial\Om,\ t>0,\\
\end{cases}
\end{equation}
time delay $\tau$ can make
the unique spatially nonhomogeneous
positive steady state unstable and induce Hopf bifurcation. Then, many authors investigated the Hopf bifurcation of models with the homogeneous Dirichlet boundary conditions, see \cite{HuR,su2009hopf,Su2011,yan2010stability,yan2012}.
Moreover, we refer to \cite {Chen2012,Chen2016,Guo2015,Guo2016} and the references therein for the Hopf bifurcation of models with the nonlocal delay effect and the homogenous Dirichlet boundary conditions.

In model \eqref{pp}, all the parameters are constant. However, due to the heterogeneity of the environment, the population may have a tendency to move up or down along the gradient of the habitats \cite{Belgacem1995}. Therefore, it is more realistic to have the following model,
\begin{equation}\label{spp}
\begin{cases}
\ds\frac{\partial u(x,t)}{\partial t} =\nabla\cdot[d\nabla u-au\nabla m]+
u(x,t)\left[m(x)-u(x,t-r)\right],&x\in\Om,\ t>0,\\
u(x,t)=0,&  x\in\partial\Om,\ t>0,\\
\end{cases}
\end{equation}
where $u(x,t)$ represents the population density at location
$x$ and time $t$, $d>0$ is the diffusion coefficient, time delay $r>0$ represents the maturation time, and $\Om$ is a bounded
domain in $\mathbb{R}^n$ ($1\le n\le 3$)
with a smooth boundary $\partial\Om$. Moreover, the intrinsic growth rate $m(x)$ is spatially dependent and may change sign, which means that, the intrinsic growth rate of the population is positive on favorable habitats and negative on unfavorable ones, and $a$ measures the tendency of the population to move up or down along the gradient of $m(x)$.
For $r=0$, Cantrell and Cosner \cite{cantrell1989,cantrell1991} investigated the effects of spatial heterogeneity on the dynamics of model \eqref{spp} for the case of $a=0$, and Belgacem and Cosner \cite{Belgacem1995} considered the case of $a\ne0$. We also refer to \cite{Cantrell1998,Chen2008,Cosner2003,He2016,Lou2006,LouC} and the references therein for the effects of spatial heterogeneity on single population and two competing populations models.

In this paper, we mainly investigate whether time delay $r$ can induce Hopf bifurcation for reaction-diffusion-advection model \eqref{spp}.
As in \cite{Belgacem1995}, letting $v=e^{(-a/d)m(x) }u$, $t=\tilde t/d$, dropping the tilde sign, and denoting $\la=1/d$, $\alpha=a/d$, $\tau=dr$, system \eqref{spp} can be transformed as follows:
\begin{equation}\label{delay}
\begin{cases}
\ds\frac{\partial v}{\partial t} =e^{-\alpha m(x)}\nabla\cdot[e^{\alpha m(x)}\nabla v]+
\la v\left[m(x)-e^{\alpha m(x)}v(x,t-\tau)\right],&x\in\Om,\ t>0,\\
v(x,t)=0,&  x\in\partial\Om,\ t>0.\\
\end{cases}
\end{equation}
Throughout the paper, unless otherwise specified, $m(x)$ satisfies the following assumption
\begin{enumerate}
\item[$\mathbf{(A_1)}$] $m(x)\in C^2(\overline\Omega)$, and $\max_{x\in\overline\Omega}m(x)>0$.
\end{enumerate}
The following eigenvalue problem
\begin{equation}\label{eigen-1}
\begin{cases}
-e^{-\alpha m(x)}\nabla\cdot[e^{\alpha m(x)}\nabla v]=-\Delta v-\alpha\nabla m\cdot\nabla v=\la  m(x) v,&x\in\Omega,\\
v(x)=0,&x\in\partial\Om,
\end{cases}
\end{equation}
is crucial to derive our main results.
It follows from \cite{Belgacem1995,CantrellB,LouC} that, under assumption $\mathbf{(A_1)}$, \eqref{eigen-1}
has a unique principal eigenvalue $\la_*>0$ admitting a strictly positive eigenfunction
$\phi\in C_0^{1+\delta}(\overline \Omega)$ for some $\delta\in(0,1)$.
Then, we can obtain the similar results as the case of spatial homogeneity \cite{busenberg1996stability,su2009hopf}:
for $\la\in(\la_*,\la^*]$, where $0<\la^*-\la_*\ll1$,
there exists a sequence of values
$\{\tau_n(\la)\}_{n=0}^{\infty}$, such that, when $\tau=\tau_n(\la)$, Eq. \eqref{delay} occurs
Hopf bifurcation at the unique spatially nonhomogeneous positive
steady state. Note that $\la=1/d$, where $d$ is the diffusion coefficient of model \eqref{spp}. Then, we see that there exists $d_*<1/\la_*$, such that
for $d\in[d_*,1/\la_*)$,
there exists a sequence of values
$\{r_n(d)\}_{n=0}^{\infty}$, such that Eq. \eqref{spp} occurs Hopf bifurcation when delay $r=r_n(d)$.

The rest of the paper is organized as follows. In Section 2, we
study the stability and Hopf bifurcation of the spatially
nonhomogeneous positive steady state for Eq. \eqref{delay}.
In Section 3, we derive an explicit formula, which can be used to determine
the direction of the Hopf bifurcation and the stability of the bifurcating periodic orbits. In Section 4, we give some remarks on the model with zero-flux boundary condition, and some numerical simulations are illustrated to support the obtained theoretical results. As in \cite{Chen2012,Chen2016},
throughout the paper, we also denote the spaces $X=H^2(\Om)\cap H^1_0(\Om)$,
$Y=L^2(\Om)$, $C=C([-\tau,0],Y)$, and $\mathcal{C}=C([-1,0],Y)$.
Moreover, we denote the complexification of a linear space $Z$ to be $Z_\mathbb{C}:= Z\oplus
iZ=\{x_1+ix_2|~x_1,x_2\in Z\}$, the domain
of a linear operator $L$ by $\mathscr{D}(L)$, the kernel of $L$ by $\mathscr{N}(L)$, and the range of $L$ by $\mathscr{R}(L)$.
For Hilbert space $Y_{\mathbb C}$, we use
the standard inner product $\langle u,v \rangle=\ds\int_{\Om} \overline
u(x) {v}(x) dx$.

\section {Stability and Hopf bifurcation}
In this section, we first consider the existence of positive steady states of Eq.
\eqref{delay}, which satisfy:
\begin{equation} \label{steady}
\begin{cases}
\ds \nabla\cdot[e^{\alpha m(x)}\nabla v]+
\la e^{\alpha m(x)}v\left[m(x)-e^{\alpha m(x)}v\right]=0, &
x\in\Om,\\
v(x)=0,&x\in\partial\Om.
\end{cases}
\end{equation}
Actually, it follows from \cite{Belgacem1995,LouC} that, for $\tau=0$, model \eqref{delay} has a unique
positive steady state which is global attractive among non-trivial nonnegative solutions if $\la>\la_*$, and the trivial steady state is global attractive if $\la\le\la_*$. Denote
\begin{equation}\label{LL}
L:=\nabla\cdot[e^{\alpha m(x)}\nabla] +\la_*e^{\alpha m(x)}m(x),
\end{equation}
where $\la_*>0$ is the unique principal eigenvalue of problem \eqref{eigen-1} admitting a strictly positive eigenfunction $\phi$.
Note that
\begin{equation*}
      X  =  \mathscr{N}\left(L\right)\oplus X_1,\;\;
      Y  = \mathscr{N}\left(L\right)\oplus Y_1,
\end{equation*}
 where
\begin{equation}\label{L}
\begin{split}
\mathscr{N}\left(L\right)=&\text{span}\{\phi\},\;\;
X_1=\left\{y\in X:\int_{\Om}\phi(x) y(x)dx=0\right\},\\
Y_1=&\mathscr{R}\left(L\right)=\left\{y\in Y:\int_{\Om}\phi(x)
y(x)dx=0\right\}.
\end{split}
\end{equation}
Then we can give a profile of the unique positive steady state near $\la_*$.
\begin{theorem}\label{Tsteady}
There exist $\la^*>\la_*$ and a continuously differential mapping
$\la\mapsto(\xi_{\la},\beta_{\la})$ from $[\la_*,\la^*]$ to
$X_1\times\mathbb{R}^+$ such that, for $\la\in(\la_*,\la^*]$,  the unique positive steady state of Eq. \eqref{delay} has the following form
\begin{equation}\label{steady1}
u_{\la}=\beta_{\la}(\la-\la_*)[\phi+(\la-\la_*)\xi_{\la}].
\end{equation}
Moreover, for $\la=\la_*$,
\begin{equation}\label{al}
\beta_{\la_*}
=\ds\f{\ds\int_{\Om}m(x)e^{\alpha m(x)}\phi^2(x)dx}{\la_*\ds\int_{\Om}e^{2\alpha m(x)}\phi^3(x)dx},
\end{equation}
and $\xi_{\la_*}\in X_1$ is the unique solution of the following equation
\begin{equation}\label{xi}
L\xi+\phi\left(m(x)e^{\alpha m(x)}-\la_*\beta_{\la_*}e^{2\alpha m(x)}\phi\right)=0,
\end{equation}
where $L$ is defined as in Eq. \eqref{LL}.
\end{theorem}
\begin{proof}
Noticing that
\begin{equation}\label{pos}
    \la_*\int_\omega m(x)e^{\alpha m(x)} \phi^2(x)dx=\int_\Omega e^{\alpha m(x)}|\nabla \phi(x)|^2dx>0,
\end{equation}
we see that $\beta_{\la_*}$ is well defined and positive. It follows that $$\phi\left(m(x)e^{\alpha m(x)}-\la_*\beta_{\la_*}e^{2\alpha m(x)}\phi\right)\in\mathscr{R}(L)=Y_1,$$ and hence $\xi_{\la_*}$ is
 well defined. Substituting $u=\beta(\la-\la_*)\left[\phi+(\la-\la_*)\xi\right]$ into Eq. \eqref{steady}, we see that $(\beta,\xi)$ satisfies
\begin{equation*}
    m(\xi,\beta,\la)=
L\xi+m(x)e^{\alpha m(x)}\left[\phi+(\la-\la_*)\xi\right]\\
        -\la\beta e^{2\alpha m(x)}[\phi+(\la-\la_*)\xi]^2=0.
\end{equation*}
Noticing that $\Omega$ is a bounded domain in $\mathbb{R}^n (1\le n\le 3)$ with a smooth boundary $\partial\Omega$, we see that $X_1$ is compactly imbedded into $C^\gamma(\overline \Omega)$ for some $\gamma\in(0,1)$, and hence $m(\xi,\beta,\la)$ is a function from $X_1\times \mathbb{R}^2$ to $Y$.
It follows from Eqs. \eqref{al} and \eqref{xi} that
$m(\xi_{\la_*},\beta_{\la_*}, \la_*)=0$, and
$$D_{(\xi,\beta)}m(\xi_{\la_*},\beta_{\la_*},\la_*)[\eta,\epsilon]=L\eta-\la_*\epsilon
e^{2\alpha m(x)}\phi^2,$$ where
$D_{(\xi,\beta)}m(\xi_{\la_*},\beta_{\la_*},\la_*)[\eta,\epsilon]$
is the Fr\'echet derivative of $m$ with respect to $(\xi,\beta)$ at $(\xi_{\la_*},\beta_{\la_*},\la_*)$.
One can easily check that $D_{(\xi,\beta)}m(\xi_{\la_*},\beta_{\la_*},\la_*)$ is
a bijection from $X_1\times\mathbb{R}$ to $Y$. Then, it follows from the implicit
function theorem that there exist $\la^*>\la_*$ and a
continuously differentiable mapping
$\la\mapsto(\xi_\la,\beta_\la)\in X_1\times\mathbb{R}^+$ such that
$$
m(\xi_\la,\beta_\la,\la)=0,\;\; \la\in[\la_*,\la^*].
$$
Therefore,
$\beta_\la(\la-\la_*)[\phi+(\la-\la_*)\xi_\la]$ is a positive solution of Eq.
\eqref{steady}.
\end{proof}
Linearizing system \eqref{delay} at $u_\la$, we have
\begin{equation}
\label{linear}\begin{cases}\ds\frac{\partial v}{\partial t} =
e^{-\alpha m(x)}\nabla\cdot[e^{\alpha m(x)}\nabla v]+
\la
\left[m(x)-e^{\alpha m(x)}u_{\la}\right]v\\~~~~-\la e^{\alpha m(x)}
u_{\la}v(x,t-\tau),& x\in\Om,\; t>0,\\
v(x,t)=0,&x\in\partial\Om,\; t>0.
\end{cases}
\end{equation}
It follows from \cite{wu1996theory} that the solution semigroup of
Eq. \eqref{linear} has the infinitesimal generator $A_\tau(\la)$
satisfying
\begin{equation}\label{Ataula}A_\tau(\la) \Psi=\dot\Psi,\end{equation}
where
\begin{equation*}\begin{split}
 \mathscr{D}(A_\tau(\la)) = \{ &\Psi\in C_\mathbb{C}
\cap C^1_\mathbb{C}:\ \Psi(0)\in X_{\mathbb{C}},\dot\Psi(0)=e^{-\alpha m(x)}\nabla\cdot[e^{\alpha m(x)}\nabla \Psi(0)]\\&+\la \left[m(x)-e^{\alpha m(x)}u_{\la}\right]\Psi(0)-\la e^{\alpha m(x)} u_{\la}\Psi(-\tau)  \},
\end{split}\end{equation*}
and $
C^1_\mathbb{C}=C^1([-\tau,0],Y_\mathbb{C})$. Moreover, $\mu\in\mathbb{C}$ is an eigenvalue of $A_\tau(\la)$, if and only if there exists $\psi(\ne0)\in
X_{\mathbb{C}}$ such that $\Delta(\la,\mu,\tau)\psi=0$,
where
\begin{equation}\label{triangle}
\begin{split}
&\Delta(\la,\mu,\tau)\psi:\\
=&e^{-\alpha m(x)}\nabla\cdot[e^{\alpha m(x)}\nabla \psi]+\la \left[m(x)-e^{\alpha m(x)}u_{\la}\right]\psi-\la e^{\alpha m(x)} u_{\la}\psi e^{-\mu\tau}-\mu\psi.
\end{split}
\end{equation}
We will show that the eigenvalues of $A_\tau(\la)$ could pass through the imaginary axis when time delay $\tau$ increases.
Actually, one can easily check that $A_\tau(\la)$ has a purely imaginary eigenvalue $\mu=i\nu\
(\nu>0)$ for some $\tau\ge0$, if and only if
\begin{equation}\label{eigen}
e^{-\alpha m(x)}\nabla\cdot[e^{\alpha m(x)}\nabla \psi]+\la \left[m(x)-e^{\alpha m(x)}u_{\la}\right]\psi-\la e^{\alpha m(x)} u_{\la}\psi e^{-i\theta}-i\nu\psi=0
\end{equation}
is solvable for some value of $\nu>0$, $\theta\in[0,2\pi)$, and $\psi(\ne 0)\in X_{\mathbb{C}}$.
First, we
give the following estimates for solutions of \eqref{eigen}.
\begin{lemma}\label{nu}
If $(\nu_\la,\theta_\la,\psi_\la)$ solves Eq.
\eqref{eigen} with $\nu_\la>0$, $\theta_\la\in[0,2\pi)$, and $\psi_\la( \ne 0) \in X_{\mathbb{C}}$, then
\begin{equation}\label{nues}
\nu_\la\int_\Omega e^{\alpha m(x)}|\psi_\la|^2 dx=\la\sin\theta_\la\int_\Omega e^{2\alpha m(x)}u_\la|\psi_\la|^2dx,
\end{equation}
and $\ds
\frac{\nu_\la}{\la-\la_*}$ is bounded for
$\la\in(\la_*,\la^*]$.\end{lemma}
\begin{proof}
Substituting $(\nu_\la,\theta_\la,\psi_\la)$ into Eq. \eqref{eigen}, multiplying \eqref{eigen} by $e^{\alpha m(x)}\overline\psi_\la$, and integrating the result over $\Omega$, we have
\begin{equation*}
\begin{split}
&\left\langle \psi_\la, \nabla\cdot[e^{\alpha m(x)}\nabla \psi_\la]\right\rangle+\la\int_\Omega\left[m(x)e^{\alpha m(x)}-e^{2\alpha m(x)}u_{\la}\right]|\psi_\la|^2dx\\
&-\la\int_\Omega e^{2\alpha m(x)}u_\la|\psi_\la|^2dx e^{-i\theta_\la}-i\nu_\la\int_\Omega e^{\alpha m(x)}|\psi_\la|^2 dx=0.
\end{split}
\end{equation*}
Noticing that
$$\left\langle \psi_\la, \nabla\cdot[e^{\alpha m(x)}\nabla \psi_\la]\right\rangle=-\int_\Omega e^{\alpha m(x)}|\nabla\psi_\la|^2dx<0,$$ we see that Eq. \eqref{nues} holds.
Therefore,
$$\ds\f{\nu_\la}{\la-\la_*}=\ds\f{\la\sin\theta_\la \int_\Omega e^{2\alpha m(x)}u_\la|\psi_\la|^2dx}{(\la-\la_*)\int_\Omega e^{\alpha m(x)}|\psi_\la|^2dx}\le \la |\beta_\la|e^{\alpha\max_\Omega m(x)}\left[\|\phi\|_\infty+(\la-\la_*)\|\xi_\la\|_\infty\right].$$
It follows from the continuity of
$\la\mapsto(\|\xi_{\la}\|_{\infty},\beta_{\la})$ that $\ds
\frac{\nu_\la}{\la-\la_*}$ is bounded for $\la\in(\la_*,\la^*]$.
\end{proof}
The following result is similar to Lemma 2.3 of \cite{busenberg1996stability} and
we omit the proof here.
\begin{lemma}\label{lem21}
If $z\in X_{\mathbb C}$ and $\langle\phi,z\rangle=0,$ then $|
\langle Lz,z\rangle|\geq
\la_2\|z\|^2_{Y_{\mathbb{C}}}$, where $\la_2$ is the second
eigenvalue of operator $-L$.
\end{lemma}

Now, for $\la\in(\la_*,\la^*]$, letting
\begin{equation}
\label{eigen2}
\begin{split}
&\psi= r\phi+(\la-\la_*)z,\;\;\;\; z\in (X_1)_{\mathbb{C}},\ \ \ \ r\geq0, \\
 &\|\psi\|^2_{Y_{\mathbb{C}}}=r^2\|\phi\|^2_{Y_{\mathbb{C}}}
 +(\la-\la_*)^2\|z\|^2_{Y_{\mathbb{C}}}=\|\phi\|^2_{Y_{\mathbb{C}}},
 \end{split}
 \end{equation}
and substituting \eqref{steady1}, \eqref{eigen2} and $\nu=(\la-\la_*)h$
into Eq. \eqref{eigen}, we see that $(\nu,\theta,\psi)$ solves Eq. \eqref{eigen}, where $\nu>0$, $\theta\in[0,2\pi)$ and $\psi\in X_{\mathbb{C}}(\|\psi\|^2_{Y_{\mathbb{C}}}=\|\phi\|^2_{Y_{\mathbb{C}}})$,
if and only if the following system:
\begin{equation}\label{g1}
\begin{cases}g_1(z,r,h,\theta,\la):=Lz-\la\beta_\la e^{2\alpha m(x)}\left[\phi+(\la-\la_*)\xi_{\la}\right]
\left[r\phi+(\la-\la_*)z\right]e^{-i\theta}\\
+[r\phi+(\la-\la_*)z]\left\{m(x)e^{\alpha m(x)}-\la \beta_\la e^{2\alpha m(x)}\left[\phi+(\la-\la_*)\xi_{\la}\right]-ihe^{\alpha m(x)}\right\}=0\\
 g_2(z,r,\la):=(r^2-1)\|\phi\|^2_{Y_{\mathbb{C}}}+(\la-\la_*)^2\|z\|^2_{Y_{\mathbb{C}}}=0
\end{cases}
\end{equation}
is solvable for some value of $z\in (X_1)_{\mathbb{C}}$, $h>0$, $r\ge0$, and $\theta\in[0,2\pi)$.
Define
$G:(X_1)_{\mathbb C}\times \mathbb{R}^4\to
Y_{\mathbb C}\times \mathbb{R}$ by $G=(g_1,g_2)$, and we find that
$G(z,r,h,\theta,\la)=0$ is uniquely solvable for $\la=\la_*$.
\begin{lemma}\label{l25}
The following equation \begin{equation}\label{3.5G}
\begin{cases}
G(z,r,h,\theta,\la_*)=0\\
z\in (X_1)_{\mathbb{C}},\;h>0\;r\ge0,\; \theta\in[0,2\pi)\\
\end{cases}
\end{equation} has a unique solution $(z_{\la_*},r_{\la_*},h_{\la_*},\theta_{\la_*})$. Here
\begin{equation}\label{lastar}
    r_{\la_*}=1,\;\;\theta_{\la_*}=\pi/2,\;\;h_{\la_*}=\ds\f{\int_\Omega m(x)e^{\alpha m(x)}\phi^2dx}{\int_\Omega e^{\alpha m(x)}\phi^2(x)dx},
\end{equation}
and $z_{\la_*}\in(X_1)_{\mathbb C}$ is the unique solution of
\begin{equation}\label{lastari}
L z=-i\la_*\beta_{\la_*}e^{2\alpha m(x)}\phi^2\\
+ih_{\la_*}e^{\alpha m(x)}\phi
-\phi\left(m(x)e^{\alpha m(x)}-\la_* \beta_{\la_*} e^{2\alpha m(x)}\phi\right),
\end{equation}
where $L$ is defined as in Eq. \eqref{LL}.
\end{lemma}
\begin{proof}
From Eq. \eqref{g1}, we see that $g_2(z,r,\la_*)=0$ if and only if $r=r_{\la_*}=1$. Note that
\begin{equation}
\begin{split}
g_1(z,r_{\la_*},h,\theta,\la_*)=&Lz-\la_*\beta_{\la_*}e^{2\alpha m(x)}\phi^2e^{-i\theta}\\
-&ihe^{\alpha m(x)}\phi
+\phi\left(m(x)e^{\alpha m(x)}-\la_* \beta_{\la_*} e^{2\alpha m(x)}\phi\right).
\end{split}
\end{equation}
Then \begin{equation*}
\begin{cases}
g_1(z,r_{\la_*},h,\theta,\la_*)=0\\
z\in (X_1)_{\mathbb{C}},\;h>0\;r\ge0,\; \theta\in[0,2\pi)\\
\end{cases}
\end{equation*} is solvable if and only if
\begin{equation}\label{cost}
\begin{cases}
\la_*\beta_{\la_*}\int_\Omega e^{2\alpha m(x)}\phi^3dx\sin\theta=h\int_\Omega e^{\alpha m(x)}\phi^2dx
\\
\la_*\beta_{\la_*}\int_\Omega e^{2\alpha m(x)}\phi^3dx\cos\theta=0\\
\end{cases}
\end{equation}
is solvable for a pair $(\theta,h)$ with $h>0$ and $\theta\in[0,2\pi)$.
This, combined with Eq. \eqref{al}, leads to
\begin{equation}\label{hthe}
\theta=\theta_{\la_*}=\pi/2,\;\;h=h_{\la_*}=\ds\f{\la_*\beta_{\la_*}\int_\Omega e^{2\alpha m(x)}\phi^3dx}{\int_\Omega e^{\alpha m(x)}\phi^2dx}=\ds\f{\int_\Omega m(x)e^{\alpha m(x)}\phi^2dx}{\int_\Omega e^{\alpha m(x)}\phi^2dx}.
\end{equation}
Consequently, $g_1(z,r_{\la_*},h_{\la_*},\theta_{\la_*},\la_*)=0$ has a unique solution $z_{\la_*}$, which satisfies Eq. \eqref{lastari}.
\end{proof}
Then we solve $G=0$ for $\la\in(\la_*,\la^*]$.
\begin{theorem}\label{cha}
There exist $\tilde \la^*>\la_*$ and a continuously differentiable mapping
$\la\mapsto(z_\la,r_\la,h_\la,\theta_\la)$ from
$[\la_*,\tilde \la^{*}]$ to $(X_1)_{\mathbb C}\times \mathbb{R}^3$ such that
$G(z_\la,r_\la,r_\la,\theta_\la,\la)=0$. Moreover, for $\la\in[\la_*,\tilde \la^{*}]$,
\begin{equation}\label{3.6G} \begin{cases}
G(z,r,h,\theta,\la)=0\\
z\in (X_1)_{\mathbb{C}},\;h,\;r\ge0, \;\theta\in[0,2\pi)\\
\end{cases}
\end{equation}
 has a unique solution $(z_\la,r_\la,h_\la,\theta_\la)$.
\end{theorem}
\begin{proof}
Let $T=(T_1,T_2):(X_1)_{\mathbb C}\times \mathbb{R}^3\mapsto
Y_{\mathbb C}\times \mathbb{R}$ be the Fr\'echet derivative of $G$ with respect to
$(z,r,h,\theta)$ at $(z_{\la_*},r_{\la_*},h_{\la_*},\theta_{\la_*},\la_*)$.
Then,
\begin{equation*}
\begin{split}
T_1(\chi,\kappa,\epsilon,\vartheta)=&
L\chi-i\epsilon e^{\alpha m(x)}\phi+\vartheta\la_*\beta_{\la_*}e^{2\alpha m(x)}\phi^2\\
+&\kappa\phi\left[m(x)e^{\alpha m(x)}-\la_*\beta_{\la_*}e^{2\alpha m(x)}\phi-ih_{\la_*}e^{\alpha m(x)}+i\la_*\beta_{\la_*}e^{2\alpha m(x)}\phi\right],\\
T_2(\kappa)=&2\kappa\|\phi\|^2_{Y_{\mathbb{C}}}.
\end{split}
\end{equation*}
One can easily check that $T$ is
a bijection from $(X_1)_{\mathbb C}\times \mathbb{R}^3$ to $Y_{\mathbb
C}\times \mathbb{R}.$ This, combined with the implicit function theorem, implies that
there exist $\tilde \la^*>\la_*$ and a continuously differentiable mapping
$\la\mapsto(z_\la,r_\la,h_\la,\theta_\la)$ from
$[\la_*,\tilde\la^{*}]$ to $X_{\mathbb C}\times \mathbb{R}^3$ such that
$G(z_\la,r_\la,h_\la,\theta_\la,\la)=0$. To prove the uniqueness, we only need to
verify that if $z\in (X_1)_{\mathbb{C}}$, $r^{\la},\;h^{\la}>0$, $\theta^{\la}\in [0,2\pi)$, and  $G(z^\la,r^\la,h^\la,\theta^\la,\la)=0$,
then
$$
(z^\la,r^\la,h^\la,\theta^\la)\rightarrow(z_{\la_*},r_{\la_*},h_{\la_*},\theta_{\la_*})
=\left(z_{\la_*},1,h_{\la_*},\frac{\pi}{2}\right)
$$
as $\la\rightarrow \la_*$ in the norm of $X_{\mathbb{C}}\times
\mathbb{R}^3.$ It follows from Lemma \ref{nu} and Eq. \eqref{g1} that $\{h^\la\}, \{r^\la\}$ and $\{\theta^\la\}$ are
bounded for $\la\in[\la_*,\tilde\la^{*}]$. Note that $\{\beta_{\la}\}$ and
$\{\xi_{\la}\}$ are bounded for $\la\in[\la_*,\tilde\la^{*}]$. As in Theorem 2.4 of \cite{busenberg1996stability}, we can obtain that there exist
$M_1,M_2>0$ such that
$$\la_2\|z^\la\|^2_{Y_{\mathbb{C}}}\leq |
\langle Lz,z\rangle|\le
M_1
\|\phi\|_{Y_{\mathbb{C}}}\|z^\la\|_{Y_{\mathbb{C}}}+
M_2(\la-\la_*)\|z^\la\|^2_{Y_{\mathbb{C}}},
$$
where $\la_2$ is defined as in Lemma \ref{lem21}.
Therefore, if $\tilde \la_*$ is sufficiently small, $\{z^\la\}$ is bounded in
$Y_{\mathbb C}$ for $\la\in[\la_*,\tilde\la^{*}]$. Since the
operator $L^{-1}$ is bounded, we see that
$\{z^\la\}$ is also bounded in $(X_1)_{\mathbb C}$, which implies that
$\{(z^\la,r^\la,h^\la,\theta^\la):\ \la\in(\la_*,\tilde\la^*]\}$ is
precompact in $Y_{\mathbb C}\times\mathbb{R}^3.$ Then, there exists
a subsequence $\{(z^{\la^n},r^{\la^n},h^{\la^n},\theta^{\la^n})\}$
 such that
$$
(z^{\la^n},r^{\la^n},h^{\la^n},\theta^{\la^n})
 \to(z^{\la_*},r^{\la_*},h^{\la_*},\theta^{\la_*})
 \;\text{in}\; Y_{\mathbb C}\times \mathbb{R}^3, \ \ \la^n\rightarrow\la_*\ \text{ as }\
 n\rightarrow\infty.
$$
Taking the limit of the equation
$L^{-1}g_1(z^{\la^n},r^{\la^n},h^{\la^n},\theta^{\la^n},\la^n)=0$
as $n\rightarrow\infty$, we see that
$G(z^{\la_*},r^{\la_*},h^{\la_*},\theta^{\la_*},\la_*)=0$. It follows from Lemma \ref{l25} that
$$(z^{\la_*},r^{\la_*},h^{\la_*},\theta^{\la_*})
=(z_{\la_*},r_{\la_*},h_{\la_*},\theta_{\la_*}).$$ This completes the proof.
\end{proof}
From Theorem \ref{cha}, we derive the following result.
\begin{theorem}\label{c25}
For each $\la\in(\la_*,\tilde \la^*],$ the
following equation
\begin{equation*}
\begin{cases}
\Delta(\la,i\nu,\tau)\psi=0\\
\nu\geq0,\;\tau\ge0,\;\psi (\ne 0) \in X_{\mathbb C}\\
\end{cases}
\end{equation*}
has a solution $(\nu,\tau,\psi)$, if and
only if
\begin{equation}\label{par}
\nu=\nu_\la=(\la-\la_*)h_\la,\;\psi= c \psi_\la,\;
\tau=\tau_{n}=\frac{\theta_\la+2n\pi}{\nu_\la},\;\; n=0,1,2,\cdots,
\end{equation}
where $\psi_\la=r_\la\phi+(\la-\la_*)z_\la$,
$c$ is a nonzero constant, and
$z_\la,r_\la,h_\la,\theta_\la$ are defined as in Theorem
\ref{cha}.
\end{theorem}
In the following, we will always assume $\la\in(\la_*,\tilde\la^*]$ for simplicity, where $0 < \la^*-\la_*\ll 1$. Actually, the value of $\tilde\la^*$ may be chosen smaller than the one in Theorem \ref{cha}, since further perturbation
arguments are used.
Now, we give some estimates to prove the simplicity of $i\nu_\la$.

\begin{lemma}\label{thm34}
Assume that $\lambda\in(\lambda_*,\tilde \lambda^*]$. Then,
for $n=0,1,2,\cdots$, \begin{equation}\label{sn}
S_{n}(\la):
=\int_{\Om}e^{\alpha m(x)}\psi_{\la}^2dx-\lambda\tau_{n}e^{-i\theta_\lambda}\int_{\Om}e^{2\alpha m(x)}u_\lambda
\psi_{\la}^2
dx\neq0,
\end{equation}
where $\psi_{\la}$ is defined as
in Theorem \ref{c25}.
\end{lemma}
\begin{proof}
It follows from Theorems \ref{cha} and \ref{c25} that $\theta_{\la}\to \pi/2$, $\tau_n(\la-\la_*)\to (\ds\f{\pi}{2}+2n\pi)/h_{\la_*}$, $\psi_\la\to\phi$ in $X_{\mathbb{C}}$ as $\la\to\la_*$. This, combined with Eq. \eqref{hthe}, yields
\begin{equation}\label{eq31}
\begin{split}
&\lim_{\la\to\la_*}S_n(\la) \\
=&\int_{\Omega} e^{\alpha m(x)}\phi^2 dx+\ds\f{i\beta_{\la_*}\la_*}{h_{\la_*}}\left(\ds\f{\pi}{2}+2n\pi\right)\int_\Omega e^{2\alpha m(x)}\phi^3dx\\
=&\left[1+i(\frac{\pi}{2}+2n\pi)\right]\int_{\Om}e^{\alpha m(x)}\phi^2dx\ne0.
\end{split}
\end{equation}
This completes the proof.
\end{proof}
Then, by virtue of Lemma 2.7, we obtain that $i\nu$ is simple as follows.
\begin{theorem}\label{thm34a}
Assume that
$\lambda\in(\lambda_*,\tilde\lambda^*]$. Then $\mu=i\nu_\lambda$ is a simple
eigenvalue of $A_{\tau_{n}}$ for $n=0,1,2,\cdots$, where $i\nu_\la$ and $\tau_n$ are defined as in Theorem \ref{c25}.
\end{theorem}
\begin{proof}
It follows from Theorem \ref{c25} that $\mathscr{N}[A_{\tau_{n}}
(\lambda)-i\nu_\lambda]=\text{Span}[e^{i\nu_\lambda\theta}\psi_\lambda]$, where $\theta\in[-\tau_n,0]$ and $\psi_\la$ is defined as in
Theorem \ref{c25}.
If
$\phi_1\in\mathscr{N}[A_{\tau_{n}}
(\lambda)-i\nu_\lambda]^2$,
then
$$
[A_{\tau_{n}}
(\lambda)-i\nu_\lambda]\phi_1\in\mathscr{N}[A_{\tau_{n}}(\lambda)-i\nu_\lambda]=
\text{Span}[e^{i\nu_\lambda\theta}\psi_\lambda].
$$
Therefore, there exists a constant $a$ such that
$$
[A_{\tau_{n}}
(\lambda)-i\nu_\lambda]\phi_1=ae^{i\nu_\lambda\theta}\psi_\lambda,
$$
which yields
\begin{equation}
 \label{eq32}
\begin{split}
\dot{\phi_1}(\theta)&=i\nu_\lambda\phi_1(\theta)+ae^{i\nu_\lambda\theta}\psi_\lambda,
\ \ \ \ \theta\in[-\tau_{n},0], \\
 \dot{\phi_1}(0)&=
 e^{-\alpha m(x)}\nabla\cdot[e^{\alpha m(x)}\nabla \phi_1(0)]\\&+\la \left[m(x)-e^{\alpha m(x)}u_{\la}\right]\phi_1(0)-\la e^{\alpha m(x)} u_{\la}\phi_1(-\tau_n).
 \end{split}
 \end{equation}
From the first equation of Eq. \eqref{eq32}, we see that
\begin{equation}
\label{eq33}
\begin{split}
\phi_1(\theta)&=\phi_1(0)e^{i\nu_\lambda\theta}+a\theta
e^{i\nu_\lambda\theta}\psi_\lambda,\\
\dot{\phi_1}(0)&=i\nu_\lambda\phi_1(0)+a\psi_\lambda.
\end{split}
\end{equation}
Eq. \eqref{eq32} and Eq. \eqref{eq33} imply that
\begin{equation}\label{provesimple}
\begin{split}
&e^{\alpha m(x)}\Delta(\lambda,i\nu_\la,\tau_{n})\phi_1(0)\\=&\nabla\cdot[e^{\alpha m(x)}\nabla \phi_1(0)]-i\nu_\la e^{\alpha m(x)}\psi_1(0)\\+&\la \left[m(x)e^{\alpha m(x)}-e^{2\alpha m(x)}u_{\la}\right]\phi_1(0)-\la e^{2\alpha m(x)} u_{\la}\phi_1(0)e^{-i\theta_\la}
\\=&ae^{\alpha m(x)}\left(\psi_{\la}-\lambda\tau_{n}u_\lambda\psi_{\la}e^{\alpha m(x)}e^{-i\theta_{\la}}\right).
\end{split}\end{equation}
Since $\Delta(\la,i\nu_\la,\tau_n)\psi_\la=0$, we have $\Delta(\la,-i\nu_\la,\tau_n)\overline\psi_\la=0$. This, combined with Eq. \eqref{provesimple}, yields
\begin{equation*}\begin{split}
0&=\left\langle e^{\alpha m(x)}\Delta(\la,-i\nu_\la,\tau_n)\overline \psi_\la,\phi_1(0)\right\rangle=\left\langle\overline\psi_\lambda,e^{\alpha m(x)}\Delta(\lambda,i\nu_\la,\tau_{n})\phi_1(0)\right\rangle\\
&=a\left(\int_{\Om}e^{\alpha m(x)}\psi_{\la}^2dx-\lambda\tau_{n}e^{-i\theta_\la}\int_{\Om}\int_{\Om}u_\lambda
\psi_\la^2 e^{2\alpha m(x)}
dx\right),
\end{split}
\end{equation*} which implies that
$a=0$ from Lemma \ref{thm34}. Therefore,
$$
\mathscr{N}[A_{\tau_{n}}(\lambda)-i\nu_\lambda]^j
=\mathscr{N}[A_{\tau_{n}}(\lambda)-i\nu_\lambda],\;\;j=
2,3,\cdots,\;\; n=0,1,2,\cdots,
$$
and $\lambda=i\nu_\lambda$ is a simple eigenvalue of
$A_{\tau_{n}}$ for $n=0,1,2,\cdots.$
\end{proof}
Note that $\mu=i\nu_{\la}$ is a simple eigenvalue of $A_{\tau_{n}}$. It follows from
the implicit function theorem that there
are a neighborhood $O_{n}\times D_{n}\times
H_{n}\subset\mathbb{R}\times\mathbb{C}\times X_{\mathbb{C}}$ of
$(\tau_{n},i\nu_\lambda,\psi_\lambda)$ and a continuously
differential function $(\mu(\tau),\psi(\tau)):O_{n}\rightarrow D_{n}\times
H_{n}$ such that for each $\tau\in O_{n}$, the only eigenvalue of
$A_\tau(\lambda)$ in $D_{n}$ is $\mu(\tau),$ and
\begin{equation}\label{eq34}
\begin{split}
&e^{\alpha m(x)}\Delta(\lambda,\mu(\tau),\tau)\psi(\tau)=\nabla\cdot[e^{\alpha m(x)}\nabla \psi(\tau)]-e^{\alpha m(x)}\mu(\tau)\psi(\tau)\\
&+\la \left[m(x)e^{\alpha m(x)}-e^{2\alpha m(x)}u_{\la}\right]\psi(\tau)-\la e^{2\alpha m(x)} u_{\la}\psi(\tau) e^{-\mu(\tau)\tau}=0.
\end{split}
\end{equation}
Moreover, $
\mu(\tau_{n})=i\nu_\lambda$, and $\psi(\tau_{n})=\psi_\lambda$.
Then we have the following transversality
condition.
\begin{theorem}\label{thm35}
Assume that
$\lambda\in(\lambda_*,\tilde \lambda^*]$. Then
$$
\frac{d{\mathcal Re}[\mu(\tau_{n})]} {d\tau}>0,\ \ \ \ n=0,1,2,\cdots.
$$
\end{theorem}
\begin{proof}
Differentiating Eq.(\ref{eq34}) with respect to $\tau$ at
$\tau=\tau_{n}$ yields
\begin{equation}\label{transcond}
\begin{split}
&\frac{d\mu(\tau_{n})}{d\tau}\left[-e^{\alpha m(x)}\psi_\la+\la\tau_n e^{2\alpha m(x)}u_\la\psi_\la
e^{-i\theta_\la}\right]
\\&+e^{\alpha m(x)}\Delta(\lambda,i\nu_\lambda,\tau_{
n})\frac{d\psi(\tau_{n})}{d\tau} +i\nu_\lambda \lambda e^{2\alpha m(x)}u_\lambda
\psi_\la e^{-i\theta_\la}=0.
\end{split}
\end{equation}
Note that \begin{equation}\label{aa}\left\langle \overline\psi_\la, e^{\alpha m(x)}\Delta(\la,i\nu_\la,\tau_n)\ds\f{\psi(\tau_n)}{d\tau}\right\rangle=\left\langle e^{\alpha m(x)}\Delta(\la,-i\nu_\la,\tau_n)\overline\psi_\la, \ds\f{\psi(\tau_n)}{d\tau}\right\rangle=0.
\end{equation}
Then, multiplying Eq. \eqref{transcond} by $\psi_\lambda$ and
integrating the result over $\Om$, we have
\begin{equation}
\label{eq35}
\begin{split}\displaystyle\frac{d\mu(\tau_{n})}{d\tau}=&\frac{i\nu_\lambda\lambda
e^{-i\theta_\lambda}\displaystyle\int_{\Om}e^{2\alpha m(x)}u_\lambda
\psi_\lambda^2dx}
{\int_{\Om}e^{\alpha m(x)}\psi_{\la}^2dx-\lambda\tau_{n}e^{-i\theta_\lambda}\int_{\Om}e^{2\alpha m(x)}u_\lambda
\psi_{\la}^2
dx}\\
=&\frac{1}{|S_{n}(\la)|^2}\bigg(i\nu_\lambda\lambda
e^{-i\theta_\lambda}\int_{\Om}e^{\alpha m(x)}\psi_\la^2dx
\displaystyle\int_{\Om}e^{2\alpha m(x)}u_\lambda
\psi_\lambda^2dx\\-&i\nu_\lambda\lambda^2
\tau_{n}\left[\int_{\Om}e^{2\alpha m(x)}u_\lambda
\psi_\lambda^2dx\right]^2\bigg).
\end{split}
\end{equation}
It follows from Eq. \eqref{hthe} and the expression of $u_\la$, $\theta_\la$, $\nu_\la$ and $\psi_\la$ that
\begin{equation*}
\lim_{\la\to\la_*}\ds\f{1}{(\la-\la_*)^2}\frac{d\mathcal{R}e[\mu(\tau_{n})]}{d\tau}
=\ds\f{h_{\la_*}^2}{\lim_{\la_\to\la_*}|{S_n}(\la)|^2}\left(\int_\Omega e^{\alpha m(x)}\phi^2dx\right)^2>0.
\end{equation*}
\end{proof}
From Theorems \ref{c25}, \ref{thm34a}
and \ref{thm35}, we have the result on the distribution of eigenvalues of $A_\tau(\lambda)$.
 \begin{theorem}\label{thm36}
For
$\lambda\in(\lambda_*,\tilde \lambda^*]$, the infinitesimal generator
$A_\tau(\lambda)$ has exactly $2(n+1)$ eigenvalues with positive
real parts when $\tau\in(\tau_{n},\tau_{n+1}],\
n=0,1,2,\cdots.$
\end{theorem}
 Then we obtain the stability and associated Hopf bifurcations of the positive steady state solution $u_{\la}$. We remark that the local Hopf bifurcation theorem for partial functional differential equations was proved in \cite{wu1996theory} (see Theorem 4.5 on page 208).
\begin{theorem}\label{thm37}For  $\lambda\in(\lambda_*,\tilde \lambda^*]$,
the positive steady state $u_\lambda$ of Eq. \eqref{delay}
is locally asymptotically stable when $\tau\in[0,\tau_{0})$, and
unstable when $\tau\in(\tau_{0},\infty)$. Moreover, when $\tau=\tau_n$,
$(n=0,1,2,\cdots)$, system \eqref{delay} occurs Hopf bifurcation at the positive steady state $u_{\la}$.
\end{theorem}

\section{The direction of the Hopf bifurcation}

In this section, we combine the
methods in
\cite{faria2001normal,faria2002normal,faria2002smoothness,hassard1981theory} to analyze the direction of the Hopf bifurcation of Eq. \eqref{delay}.
Letting $U(t)=u(\cdot,t)-u_\la$, $t=\tau\tilde t$,
$\tau=\tau_n+\gamma$, and dropping the tilde sign, system \eqref{delay} can be transformed as follows:
\begin{equation}\label{ab}
\ds\frac{dU(t)}{dt}=\tau_n e^{-\alpha m(x)}\nabla\cdot[e^{\alpha m(x)}\nabla U(t)]+\tau_nL_0(U_t)+J(U_t,\gamma),
\end{equation}
where $U_t\in \mathcal{C}=C([-1,0],Y)$, and
\begin{equation*}
\begin{split}
&L_{0}(U_t)=\la \left[m(x)-e^{\alpha m(x)}u_{\la}\right]U(t)-\la e^{\alpha m(x)} u_{\la}U(t-1),\\
 &J(U_t,\gamma)=\gamma \tau_n e^{-\alpha m(x)}\nabla\cdot[e^{\alpha m(x)}\nabla U(t)]+\gamma
L_{0}(U_t)-(\gamma+\tau_n)\la e^{\alpha m(x)} U(t)U(t-1).\\
\end{split}
\end{equation*}
Then Eq. \eqref{ab} occurs Hopf bifurcation near the zero equilibrium  when $\gamma=0$. Let
$\mathcal {A}_{\tau_n}$ be the infinitesimal generator of the
linearized equation
\begin{equation}\label{lab}
\ds\frac{dU(t)}{dt}=\tau_n e^{-\alpha m(x)}\nabla\cdot[e^{\alpha m(x)}\nabla U(t)]+\tau_nL_0(U_t).
\end{equation}
It follows from \cite{wu1996theory} that \begin{equation*}\begin{split}
\mathcal{A}_{\tau_n} \Psi=&\dot\Psi,\\
\mathscr{D}(\mathcal{A}_{\tau_n})=&\Big\{\Psi\in \mathcal
{C}_\mathbb{C}
\cap \mathcal {C}^1_\mathbb{C}:\ \Psi(0)\in X_{\mathbb{C}},\dot\Psi(0)=\tau_n e^{-\alpha m(x)}\nabla\cdot[e^{\alpha m(x)}\nabla \Psi(0)]\\&+\la \tau_n\left[m(x)-e^{\alpha m(x)}u_{\la}\right]\Psi(0)-\la\tau_n e^{\alpha m(x)} u_{\la}\Psi(-1)\Big\},
\end{split}\end{equation*}
where $\mathcal{C}^1_\mathbb{C}=C^1([-1,0],Y_\mathbb{C})$,
and Eq. \eqref{ab} can be written in the following abstract form
\begin{equation}\label{abab}
\ds\frac{dU_t}{dt}=\mathcal{A}_{\tau_n}U_t+X_0J(U_t,\gamma),
\end{equation}
where
\begin{equation*}X_0(\theta)=\begin{cases}0, \;\;\; & \theta\in[-1,0),\\
I, \;\;\; &\theta=0.\\
\end{cases}\end{equation*}
It follows from Theorem \ref{thm36} that $\mathcal {A}_{\tau_n}$ has
only one pair of purely imaginary eigenvalues $\pm i \nu_\la\tau_n$,
which are simple, and the corresponding eigenfunction with respect to
$i\nu_\la\tau_n$ (respectively, $-i\nu_\la\tau_n$) is $\psi_\la e^{ i \nu_\la\tau_n\theta}$
(respectively, $\overline{\psi_\la} e^{-i\nu_\la\tau_n\theta}$) for $\theta\in[-1,0]$, where
$\psi_\la$ is defined as in Theorem \ref{c25}.

Following \cite{faria2002normal,Su2011}, we introduce the formal
duality $\langle\langle\cdot,\cdot\rangle\rangle$ in $\mathcal{C}$  by
\begin{equation}\label{bil}
\langle\langle\tilde\Psi,\Psi\rangle\rangle=\langle
\tilde\Psi(0),\Psi(0)\rangle_1-\la\tau_n\int_{-1}^0\left\langle\tilde\Psi(s+1),u_\la e^{\alpha m(x)}\Psi(s) \right\rangle_1 ds,
\end{equation}
for $\Psi\in \mathcal{C}_{\mathbb{C}}$ and $\tilde\Psi\in
\mathcal{C}_{\mathbb{C}}^*:= C([0,1],Y_{\mathbb{C}})$, where $$\langle u,v \rangle_1=\ds\int_{\Om} e^{\alpha m(x)}\overline
u(x) {v}(x) dx.$$
Since $m(x)$ is bounded and $e^{\alpha m(x)}$ is positive, we see that $Y_{\mathbb{C}}$ is also a Hilbert space with this product, and
$$e^{\alpha \min_\Omega m(x)}\langle v, v\rangle \le \langle v,v\rangle_1\le e^{\alpha \max_\Omega m(x)}\langle v, v\rangle.$$
As in \cite{Hale1971}, we can compute the
formal adjoint operator $\mathcal{A}^*_{\tau_n}$ of $\mathcal{A}_{\tau_n}$ with respect to the formal duality.

\begin{lemma}\label{dualoperator}
The
formal adjoint operator $\mathcal{A}^*_{\tau_n}$ of $\mathcal{A}_{\tau_n}$ is defined by
$$\mathcal{A}^*_{\tau_n}\tilde\Psi(s)=-\dot{\tilde\Psi}(s),$$ and the domain
\begin{equation*}\begin{split}
 \mathscr{D}(\mathcal{A}^*_{\tau_n}) = \Big\{ &\tilde\Psi\in \mathcal{C}^*_\mathbb{C}
\cap (\mathcal{C}^*_\mathbb{C})^1:\tilde\Psi(0)\in X_\mathbb{C},-\dot{\tilde\Psi}(0)=\tau_n e^{-\alpha m(x)}\nabla\cdot[e^{\alpha m(x)}\nabla \tilde\Psi(0)]\\&+\la \tau_n\left[m(x)-e^{\alpha m(x)}u_{\la}\right]\tilde \Psi(0)-\la\tau_n e^{\alpha m(x)} u_{\la}\tilde\Psi(1) \Big\},
\end{split}\end{equation*}
where $(\mathcal{C}^*_\mathbb{C})^1=C^1([0,1],Y_\mathbb{C})$. Moreover,
$\mathcal{A}^*_{\tau_n}$ and $\mathcal{A}_{\tau_n}$ satisfy
\begin{equation}\label{Atauadjont}
\langle\langle
\mathcal{A}^*_{\tau_n}\tilde\Psi,\Psi\rangle\rangle=\langle\langle
\tilde\Psi,\mathcal{A}_{\tau_n}\Psi\rangle\rangle\;\;\text{for
}\Psi\in\mathscr{D}(\mathcal{A}_{\tau_n})\text{ and }\tilde\Psi\in
\mathscr{D}(\mathcal{A}^*_{\tau_n}).
\end{equation}
\end{lemma}
\begin{proof}
For $\Psi\in\mathscr{D}(\mathcal{A}_{\tau_n})$ and $\tilde\Psi\in
\mathscr{D}(\mathcal{A}^*_{\tau_n})$,
\begin{equation*}
\begin{split}
\langle\langle
\tilde\Psi,\mathcal{A}_{\tau_n}\Psi\rangle\rangle=&\left\langle
\tilde\Psi(0),
(\mathcal{A}_{\tau_n}\Psi)(0)\right\rangle_1-\la\tau_n\int_{-1}^0\left\langle\tilde\Psi(s+1),
u_\la e^{\alpha m(x)}\dot\Psi(s)\right\rangle_1 ds\\
=&\left\langle\tilde\Psi(0),\tau_n e^{-\alpha m(x)}\nabla\cdot[e^{\alpha m(x)}\nabla \Psi(0)]\right\rangle_1-\la\tau_n\left[\left\langle\tilde\Psi(s+1), u_\la e^{\alpha m(x)}\Psi(s)\right\rangle_1\right]_{-1}^{0}\\+&\left\langle \tilde\Psi(0),\la \tau_n\left[m(x)-e^{\alpha m(x)}u_{\la}\right] \Psi(0)-\la\tau_n e^{\alpha m(x)} u_{\la}\Psi(-1)\right\rangle_1\\
+&\la\tau_n\int_{-1}^0\left\langle\dot{\tilde\Psi}(s+1),
u_\la e^{\alpha m(x)}\Psi(s)\right\rangle_1 ds\\
=&\left\langle
(\mathcal{A}^*_{\tau_n}\tilde\Psi)(0),\Psi(0)\right\rangle_1
-\la\tau_n\int_{-1}^0\left\langle-\dot{\tilde\Psi}(s+1),
u_\la e^{\alpha m(x)}\Psi(s)\right\rangle_1 ds\\
=&\langle\langle
\mathcal{A}^*_{\tau_n}\tilde\Psi,\Psi\rangle\rangle.
\end{split}
\end{equation*}
\end{proof}
Similarly, it follows from Theorem \ref{thm36} that
the operator $\mathcal {A}^*_{\tau_n}$ has only one pair of purely imaginary
eigenvalues $\pm i \nu_\la\tau_n$, which are simple, and the associated
eigenfunction with respect to $-i\nu_\la\tau_n$ (respectively, $i\nu_\la\tau_n$) is $\overline\psi_\la
e^{i \nu_\la\tau_n s}$ (respectively,
$\psi_\la e^{-i \nu_\la\tau_ns}$) for $s\in[0,1]$,
where $\psi_\la$ is defined as in Theorem \ref{c25}.
From \cite{wu1996theory}, we see that the center subspace of Eq. \eqref{ab} is
$P=\text{span}\{p(\theta),\overline {p}(\theta)\}$, where
$p(\theta)=\psi_\la e^{ i \nu_\la\tau_n\theta}$ is the eigenfunction
of $\mathcal{A}_{\tau_n}$ with respect to $i\nu_\la\tau_n$.
The formal adjoint subspace of $P$ is $P^*=\text{span}\{q(s),\overline{
q}(s)\}$, where $q(s)=\overline\psi_\la e^{ i \nu_\la\tau_ns}$ is the
eigenfunction of $\mathcal{A}^*_{\tau_n}$ with respect to
$-i\nu_\la\tau_n$. Let $\Phi_p=(p(\theta),\overline p(\theta))$,
$\Psi_P=\ds\f{1}{\overline {S_n}(\la)}(q(s),\overline {q}(s))^{T}$,
where $S_n(\la)$ is defined in Lemma \ref{thm34}, and one can easily check that
$\langle\langle\Psi_p,\Phi_p\rangle\rangle= I$, where $I$ is the
identity matrix in $\mathbb{R}^{2\times2}$. Moreover, $\mathcal{C}_{\mathbb{C}}$ can be decomposed
as $\mathcal{C}_{\mathbb{C}}=P\oplus Q$, where
$$
Q=\{\Psi\in\mathcal{C}_{\mathbb{C}}:
\langle\langle\tilde\Psi,\Psi\rangle\rangle=0 \text{ for all
}\tilde\Psi\in P^*\}.$$

Since the formulas of Hopf bifurcation are all relative to $\gamma=0$ only, we set $\gamma=0$ in
Eq. \eqref{ab}. Let
\begin{equation}\label{center}w(z,\overline
z)=w_{20}(\theta)\ds\frac{z^{2}}{2}+w_{11}(\theta)z\overline
z+w_{02}(\theta)\ds\frac{\overline z^{2}}{2}+\cdots\end{equation}
be the center manifold with the range in $Q$, and then the flow of Eq. \eqref{ab} on the center
manifold can be written as:
\begin{equation*}U_t=\Phi_p\cdot (z(t),\overline z(t))^{T}+w(z(t),\overline z(t)),\end{equation*}
where
\begin{equation}\label{z(t)}
\begin{split}
\dot{z}(t) =&\ds\f{d}{dt}\langle\langle q(s),U_t\rangle\rangle\\
=&\langle\langle q(s),
\mathcal{A}_{\tau_n}U_t\rangle\rangle+\ds\f{1}{S_n(\la)}\langle\langle q(s), X_0J(U_t,0)\rangle\rangle\\
 =& i\nu_\la\tau_n z(t)+\ds\f{1}{S_n(\la)}\left\langle q(0),
J\left(\Phi_p(z(t),\overline z(t))^{T}+w(z(t),\overline
z(t)),0\right)\right\rangle_1\\
=&i\nu_\la\tau_nz(t)+g(z,\overline z).
\end{split}
\end{equation}
Then,
\begin{equation} \label{g}
\begin{split}
g(z,\overline z)=&\ds\f{1}{S_n(\la)}\left\langle q(0),
J\left(\Phi_p(z(t),\overline z(t))^{T}+w(z(t),\overline
z(t)),0\right)\right\rangle_1\\=&g_{20}\ds\frac{z^{2}}{2}+g_{11}z\overline
z+g_{02}\ds\frac{\overline z^{2}}{2}+g_{21}\ds\frac{z^{2}\overline
z}{2}+\cdots,
\end{split}
\end{equation}
and an easy calculation implies that
\begin{equation}\label{gij}
\begin{split}
g_{20}=&-\ds\f{2\la\tau_n}{S_n(\la)}e^{-i\nu_\la\tau_n}\int_\Om e^{2\alpha m(x)}\psi^3_\la dx,\\
g_{11}=&-\left[\ds\f{\la\tau_n}{S_n(\la)}(e^{i\nu_\la\tau_n}+e^{-i\nu_\la\tau_n})\right]\int_\Om e^{2\alpha m(x)}\psi_\la|\psi_\la|^2dx,\\
g_{02}=&-\ds\f{2\la\tau_n}{S_n(\la)}e^{i\nu_\la\tau_n}\int_\Om e^{2\alpha m(x)}\psi_\la\overline{\psi}^2_\la dx,\\
g_{21}=&-\ds\f{2\la\tau_n}{S_n(\la)}\int_\Om e^{2\alpha m(x)}\psi_\la^2 w_{11}(-1)dx
-\ds\f{\la\tau_n}{S_n(\la)}\int_\Om e^{2\alpha m(x)}|\psi_\la|^2 w_{20}(-1)dx\\
-&\ds\f{\la\tau_n}{S_n(\la)}
e^{i\nu_\la\tau_n}\int_\Om e^{2\alpha m(x)}|\psi_\la|^2 w_{20}(0)dx
-\ds\f{2\la\tau_n}{S_n(\la)}e^{-i\nu_\la\tau_n}\int_\Om e^{2\alpha m(x)} \psi_\la^2 w_{11}(0)dx.
\end{split}
\end{equation}
To compute $g_{21}$, we need to compute $w_{20}(\theta)$
and $w_{11}(\theta)$ in the following.
As in \cite{Chen2012,hassard1981theory}, we see that $w_{20}(\theta)$
and $w_{11}(\theta)$ satisfy
\begin{equation}\label{ws}
\begin{cases}
(2i\nu_\la\tau_n-\mathcal{A}_{\tau_n})w_{20}=H_{20},\\
-\mathcal{A}_{\tau_n}w_{11}=H_{11}.\\
\end{cases}
\end{equation}
Here, for $-1\le\theta<0$,
\begin{equation}\label{H20}
H_{20}(\theta)=-(g_{20}p(\theta)+\overline g_{02}\overline
p(\theta)),
\end{equation}
\begin{equation}\label{H11}
H_{11}(\theta)=-(g_{11}p(\theta)+\overline g_{11}\overline
p(\theta)),
\end{equation}
and, for $\theta=0$,
\begin{equation}\label{H200}
H_{20}(0) =-\left(g_{20}p(0)+\overline g_{02}\overline
p(0)\right)-2\la\tau_ne^{-i\nu_\la\tau_n}e^{\alpha m(x)}\psi^2_\la,
\end{equation}
\begin{equation}\label{H110}
H_{11}(0) =-\left(g_{11}p(0)+\overline g_{11}\overline
p(0)\right)-\la\tau_n\left(e^{-i\nu_\la\tau_n}+e^{i\nu_\la\tau_n}\right)e^{\alpha m(x)}|\psi_\la|^2.
\end{equation}
It follows from Eqs. \eqref{ws}-\eqref{H11} that $w_{20}(\theta)$ and $w_{11}(\theta)$ can be solved as follows:
\begin{equation}\label{W20}w_{20}(\theta)=\ds\frac{ig_{20}}{\nu_\la\tau_n}p(\theta)+
\ds\frac{i\overline g_{02}}{3\nu_\la\tau_n}\overline
p(\theta)+Ee^{2i\nu_\la\tau_n\theta},
\end{equation}
and
\begin{equation}\label{W11}
w_{11}(\theta)=-\ds\frac{ig_{11}}{\nu_\la\tau_n}p(\theta)+\ds\frac{i\overline
g_{11}}{\nu_\la\tau_n}\overline p(\theta)+F.
\end{equation}
From Eq. \eqref{ws} with $\theta=0$, the definition
of $\mathcal{A}_{\tau_n}$ and
we see that $E$ satisfies
\begin{equation*}(2i\nu_\la\tau_n-\mathcal{A}_{\tau_n})
Ee^{2i\nu_\la\tau_n\theta}\bigg\vert_{\theta=0}=-2\la\tau_ne^{-i\nu_\la\tau_n}e^{\alpha m(x)}\psi^2_\la,
\end{equation*}
or equivalently,
\begin{equation}\label{E}
\Delta(\la,2i\nu_\la,\tau_n)E=2\la e^{-i\nu_\la\tau_n}e^{\alpha m(x)}\psi^2_\la.
\end{equation}
Note that $2i\nu_\la$ is not the
eigenvalue of $A_{\tau_n}(\la)$ for $\la\in(\la_*,\tilde\la^*]$, and hence
\begin{equation*}
E=2\la
e^{-i\nu_\la\tau_n}\Delta(\la,2i\nu_\la,\tau_n)^{-1}\left(e^{\alpha m(x)}\psi^2_\la\right).
\end{equation*}
Similarly, from Eqs. \eqref{ws}, \eqref{H110}, and \eqref{W11}, we have
\begin{equation}\label{F1}
F=\la\left(e^{-i\nu_\la\tau_n}+e^{i\nu_\la\tau_n}\right)\Delta(\la,0,\tau_n)^{-1}\left(
e^{\alpha m(x)}|\psi_\la|^2\right).
\end{equation}
In the following, we obtain the similar result as in \cite{Chen2012} for the expression of $E$ and $F$.
\begin{lemma}\label{comf}
Assume that $E$ and $F$ satisfy \eqref{E} and \eqref{F1}, respectively.
Then
\begin{equation}\label{EFgostar}
E=\ds\f{1}{\la-\la_*}(c_\la u_\la+\eta_{\la}),\;\; F=\ds\f{\tilde{\eta}_{\la}}{\la-\la_*},
\end{equation}
where $u_\la$ is defined as in \eqref{steady1}, $\eta_{\la}$ and $\tilde{\eta}_{\la}$ satisfy
$$\langle u_\la,\eta_{\la}\rangle=0,\;\;\lim_{\la\to\la_*}\|\eta_{\la}\|_{Y_\mathbb{C}}=0,
\;\;\lim_{\la\to\la_*}\|\tilde{\eta}_{\la}\|_{Y_\mathbb{C}}=0,$$
and the constant $c_{\la}$ satisfies $\ds \lim_{\la\to\la_*}(\la-\la_*)
c_\la=\ds\f{2i}{\alpha^2_{\la_*}(2i-1)}$.
\end{lemma}
\begin{proof} We just prove the estimate for $E$, and that for $F$ can be derived similarly.
Denote the operator
\begin{equation}\label{Lla}
L_\la:=\nabla\cdot \left[e^{\alpha m(x)}\nabla \right]+\la e^{\alpha m(x)}[m(x)-e^{\alpha m(x)}u_\la],
\end{equation}
and consequently $L_\la u_\la=0$.
Substituting $E$, defined as in Eq. \eqref{EFgostar}, into Eq. \eqref{E}, one can easily have
\begin{equation}\label{cephi}
\begin{split}&L_\la \eta_{\la}-\la e^{-2i\nu_\la\tau_n} e^{2\alpha m(x)}
u_\la(c_\la
u_\la+\eta_{\la})-2i\nu_\la e^{\alpha m(x)}(c_\la
u_\la+\eta_{\la})\\=&2(\la-\la_*)\la
e^{-i\nu_\la\tau_n}e^{2\alpha m(x)}\psi^2_\la.\end{split}
\end{equation}
Multiplying Eq. \eqref{cephi} by $u_\la$, and integrating the result over $\Omega$, we have
\begin{equation}\label{cephi2}
\begin{split}
&c_\la\left( \la e^{-2i\nu_\la\tau_n}\int_\Om
e^{2\alpha m(x)}u_\la^3dx+2i\nu_\la\int_\Omega e^{\alpha m(x)}u^2_\la dx\right)\\
=&-\la e^{-2i\nu_\la\tau_n}\int_\Om e^{2 \alpha m(x)}u^2_\la\eta_{\la}dx-2i \nu_\la \int_\Omega e^{\alpha m(x)} u_\la\eta_\la dx\\-&2\la
e^{-i\nu_\la\tau_n}(\la-\la_*)\int_\Om
e^{2\alpha m(x)}u_\la\psi^2_\la dx.
\end{split}
\end{equation}
Multiplying Eq. \eqref{cephi} by $\overline\eta_\la$, and integrating the result over $\Omega$, we obtain
\begin{equation}\label{cephi3}
\begin{split}
&\langle \eta_{\la}, L_\la\eta_{\la}\rangle -\la c_\la\int_\Om  e^{2\alpha m(x)}\overline\eta_{\la} u^2_\la dx e^{-2i\nu_\la\tau_n}-2i\nu_\la c_\la\int_\Om e^{\alpha m(x)}u_\la\overline\eta_\la dx\\
=&\la\int_\Om e^{2\alpha m(x)}u_\la|\eta_{\la}|^2 dxe^{-2i\nu_\la\tau_n}+2i\nu_\la\int_\Om e^{\alpha m(x)} |\eta_{\la}|^2 dx\\+&2\la
e^{-i\nu_\la\tau_n}(\la-\la_*)\int_\Om e^{2\alpha m(x)}\overline\eta_{\la}\psi^2_\la dx.
\end{split}
\end{equation}
It follows from the expression of $\nu_\la$, $u_\la$, $\psi_\la$ and $\tau_n$ that
\begin{equation}\label{lst}
\begin{split}
&\psi_\la\to\phi,\;\;u_\la/(\la-\la_*)\to \beta_{\la_*}\phi \;\;\text{in}\;\;  C(\overline \Omega),\\
&\nu_\la/(\la-\la_*)\to h_{\la_*},\;\;
\nu_\la\tau_n\to \ds\f{\pi}{2}+2n\pi.
\end{split}
\end{equation}
From Eqs.
\eqref{cephi2} and \eqref{lst}, we see that there exist constants $\tilde\la>\la_*$ and $M_0, M_1>0$
such that for, any $\la\in(\la_*,\tilde\la)$,
\begin{equation}
\begin{split}
&|(\la-\la_*)c_\la|\le
M_0\|\eta_{\la}\|_{Y_\mathbb{C}}+M_1.
\end{split}
\end{equation}
This, combined with Eqs. \eqref{cephi3} and \eqref{lst}, implies that there exist constants $M_2,~M_3>0$
such that for any $\la\in(\la_*,\tilde\la)$,
\begin{equation*}
\begin{split}
|\la_2(\la)|\cdot \|\eta_{\la}\|_{Y_{\mathbb{C}}}^2\le
(\la-\la_*)M_2\|\eta_{\la}\|^2_{Y_\mathbb{C}}+M_3(\la-\la_*)\|\eta_{\la}\|_{Y_\mathbb{C}},
\end{split}
\end{equation*}
where $\la_2(\la)$ is the second eigenvalue of $-L_\la$. Since
$\lim_{\la\to\la_*}\la_{2}(\la)=\la_{2}>0$, where $\la_{2}$, defined as in Lemma \ref{lem21}, is
the second eigenvalue of $-L$, we have
$\lim_{\la\to\la_*}\|\eta_{\la}\|_{Y_\mathbb{C}}=0$. This, together with
\eqref{cephi2}, implies
$$\lim_{\la\to\la_*}(\la-\la_*)c_\la=\ds\f{2i}{\beta^2_{\la_*}(2i-1)}.$$
\end{proof}

Therefore, by similar arguments to \cite{Chen2012}, one can easily check
\begin{equation}\label{glas}
\begin{split}
&\lim_{\la\to \la_*}(\la-\la_*)g_{11}=0,\\
&\lim_{\la\to
\la_*}\mathcal{R}e[(\la-\la_*)^2g_{21}]<0.
\end{split}
\end{equation}
It is well-known that the real part of the following quantity determines the direction and
stability of bifurcating periodic orbits (see \cite{hassard1981theory,wu1996theory}):
\begin{equation*}
C_1(0)=\dfrac{i}{2\nu_\la\tau_n}\left(g_{11}g_{20}-2|g_{11}|^2
-\dfrac{|g_{02}|^2}{3}\right)+\dfrac{g_{21}}{2}.
\end{equation*}
It follows from Eq. \eqref{glas} that $\lim_{\la\to
\la_*}\mathcal{R}e[(\la-\la_*)^2C_1(0)]<0$. Hence we have the
following result.
\begin{theorem}\label{T3.3}
For $\lambda\in(\lambda_*,\lambda^*]$, where $\la^*-\la_*\ll 1$, let
$\tau_n(\la)$ be the Hopf bifurcation points of Eq. \eqref{delay} obtained in Theorem \ref{c25}.
Then for each $n\in \mathbb N \cup \{0\}$, the direction of the
Hopf bifurcation at $\tau=\tau_n$ is forward and the bifurcating periodic solution
from $\tau=\tau_0$ is orbitally asymptotically stable.
\end{theorem}

\section{No-flux boundary condition and simulation}

In this section, we discussion model \eqref{spp} with no-flux boundary condition, that is,
\begin{equation}\label{sppn}
\begin{cases}
\ds\frac{\partial u(x,t)}{\partial t} =\nabla\cdot[d\nabla u-au\nabla m]+
u(x,t)\left[m(x)-u(x,t-r)\right],&x\in\Om,\ t>0,\\
d\partial_n u-au\partial_n m=0&  x\in\partial\Om,\ t>0,\\
\end{cases}
\end{equation}
where $n$ is the outward unit normal vector on $\partial\Omega$, and $\partial_n u=\nabla u \cdot n$.
As in Eq. \eqref{spp}, we also derive an equivalent model of Eq. \eqref{sppn} as follows:
\begin{equation}\label{delay2}
\begin{cases}
\ds\frac{\partial v}{\partial t} =e^{-\alpha m(x)}\nabla\cdot[e^{\alpha m(x)}\nabla v]+
\la v\left[m(x)-e^{\alpha m(x)}v(x,t-\tau)\right],&x\in\Om,\ t>0,\\
\partial_nv=0,&  x\in\partial\Om,\ t>0.\\
\end{cases}
\end{equation}
Here $m(x)$ satisfies the following assumption:
\begin{enumerate}
\item[$\mathbf{(A_2)}$] $m(x)\in C^2(\overline\Omega)$, $\max_{x\in\overline\Omega}m(x)>0$, and $\int_\Omega m(x) e^{\alpha m(x)}dx<0$; or
    \item[$\mathbf{(A_3)}$] $m(x)\in C^2(\overline\Omega)$, and $\int_\Omega m(x) e^{\alpha m(x)}dx>0$.
\end{enumerate}
Then the following discussion is divided into two cases.
\subsection{Case I}

In this case, $m(x)$ satisfies assumption $\mathbf{(A_2)}$. The method used for this case is similar to that for Dirichlet problem \eqref{delay}. In fact, it follows from \cite{Belgacem1995} that the following problem \begin{equation}\label{prp}
\begin{cases}
-e^{-\alpha m(x)}\nabla\cdot[e^{\alpha m(x)}\nabla v]=-\Delta v-\alpha\nabla m\cdot\nabla v=\la  m(x) v,&x\in\Omega,\\
\partial_n v=0,&x\in\partial\Om,
\end{cases}
\end{equation}
has a unique positive principal eigenvalue $\la_*$, and  model \eqref{delay2} admits a unique positive steady state $u_\la$ for $\la>\la_*$, if $m(x)$ satisfies assumption $\mathbf{(A_2)}$.
Moreover, we comment that the relation between $\la_*$ and $\alpha$ was also investigated in \cite{Cosner2003}: if
$\int_\Omega m(x)dx\ge0$, then $\la_*(\alpha)=0$ for all $\alpha\ge0$; and if $m(x)$ change sign and $\int_\Omega m(x)dx<0$, then there is a unique $\alpha_*>0$ such that $\la_*(\alpha)>0$ for $0<\alpha<\alpha_*$, and $\la_*(\alpha)=0$ for $\alpha>\alpha_*$.

Then, by similar arguments to Sections 2 and 3, we have the following results on model
\eqref{delay2}.
\begin{theorem}
Assume that $m(x)$ satisfies assumption $\mathbf{(A_2)}$. Then, for $\lambda\in(\lambda_*,\lambda^*]$, where $\la^*-\la_*\ll 1$, there exists a sequence $\{\tau_n\}_{n=0}^\infty$ such that
the positive steady state $u_\lambda$ of Eq. \eqref{delay2}
is locally asymptotically stable when $\tau\in[0,\tau_{0})$,
unstable when $\tau\in(\tau_{0},\infty)$, and system \eqref{delay2} occurs Hopf bifurcation at the positive steady state $u_{\la}$ when $\tau=\tau_n$,
$(n=0,1,2,\cdots)$. Moreover, the direction of the Hopf bifurcation at $\tau=\tau_n$ is forward and the bifurcating periodic solution from $\tau=\tau_0$ is orbitally asymptotically stable.
\end{theorem}

\subsection{Case II}

Note that assumption  $\mathbf{(A_2)}$  is equivalent to $m(x)$ changing sign,  $\int_\Omega m(x)dx<0$ and $\alpha<\alpha_*$.
Thus $\la_*(\alpha)>0$ under assumption $\mathbf{(A_2)}$. It will be of interest to study the dynamics of system \eqref{delay2}
for $\alpha>\alpha_*$, i.e. to understand the joint effect of strong advection and time delay.
Therefore, in this subsection, we consider the case that $m(x)$ satisfies assumption $\mathbf{(A_3)}$.
It follows from \cite{Cantrell1996,Cosner2003} that, under assumption $\mathbf{(A_3)}$,
the unique positive principal eigenvalue $\la_*(\alpha)$ of problem \eqref{prp} is zero, and the corresponding eigenfunction $\phi$ is constant.
Moreover, for any $\la>0$, system \eqref{delay}
has a unique positive steady state $u_\la$, which is globally asymptotically stable, and $u_\la$ satisfies
\begin{equation}\label{imd}
\lim_{\la\to0}u_\la(x)=\overline m:=\ds\f{\int_\Omega m(x)e^{\alpha m(x)}dx}{\int_\Omega e^{2\alpha m(x)} dx} \;\;\text{  in  }\;\;C^{1+\delta}(\overline\Omega)
\end{equation}
for some $\delta\in(0,1)$. Let $u_0(x)=\overline m$, and then $\la\to u_\la$ is continuous from $[0,\infty)$ to $C^{1+\delta}(\overline\Omega)$.
For simplicity, we choose $\phi\equiv \overline m$, and then $L$, $X_1$ and $Y_1$ (defined in Eqs. \eqref{LL} and \eqref{L}) have the following forms:
\begin{equation*}
\begin{split}
L=&\nabla\cdot[e^{\alpha m(x)}\nabla],\\
X_1=&\left\{y\in X:\int_{\Om} y(x)dx=0\right\},\\
Y_1=&\mathscr{R}\left(L\right)=\left\{y\in Y:\int_{\Om}
y(x)dx=0\right\}.
\end{split}
\end{equation*}

In order to analyze eigenvalue problem \eqref{eigen}, we first
give the following estimates for solutions of \eqref{eigen}.
\begin{lemma}\label{nu2}Assume that $\la\in(0,\la^*]$.
If $(\nu_\la,\theta_\la,\psi_\la)$ solves Eq.
\eqref{eigen} with $\nu_\la>0$, $\theta_\la\in[0,2\pi)$, and $\psi_\la( \ne 0) \in X_{\mathbb{C}}$, then
$\nu_\la/\la$ is bounded for
$\la\in(0,\la^*]$.\end{lemma}
\begin{proof}
It follows from Eq. \eqref{nues} that
$$\nu_\la/\la=\ds\f{\sin\theta_\la \int_\Omega e^{2\alpha m(x)}u_\la|\psi_\la|^2dx}{\int_\Omega e^{\alpha m(x)}|\psi_\la|^2dx}\le e^{\alpha\max_\Omega m(x)}\|u_\la\|_\infty.$$
Then, from the continuity of
$\la\mapsto\|u_{\la}\|_{\infty}$, we see that $\ds
\nu_\la/\la$ is bounded for $\la\in(0,\la^*]$.
\end{proof}
We remark that Lemma \ref{lem21} still holds for the case that $L=\nabla\cdot[e^{\alpha m(x)}\nabla]$.
Now, for $\la\in(0,\la^*]$, letting
\begin{equation}
\label{eigen22}
\begin{split}
&\psi= r\overline m+\la z,\;\;\;\; z\in (X_1)_{\mathbb{C}},\ \ \ \ r\geq0, \\
 &\|\psi\|^2_{Y_{\mathbb{C}}}=r^2\overline m^2|\Omega|
 +\la^2\|z\|^2_{Y_{\mathbb{C}}}=\overline m^2|\Omega|,
 \end{split}
 \end{equation}
and substituting \eqref{eigen22} and $\nu=\la h$
into Eq. \eqref{eigen}, we see that $(\nu,\theta,\psi)$ solves Eq. \eqref{eigen}, where $\nu>0$, $\theta\in[0,2\pi)$ and $\psi\in X_{\mathbb{C}}(\|\psi\|^2_{Y_{\mathbb{C}}}=\|\phi\|^2_{Y_{\mathbb{C}}})$,
if and only if the following system:
\begin{equation}\label{g12}
\begin{cases}\tilde g_1(z,r,h,\theta,\la):=\nabla\cdot[e^{\alpha m(x)}\nabla z]+e^{\alpha m(x)}\left[m(x)-e^{\alpha m(x)}u_\la\right](r\overline m+\la z)\\
-e^{2\alpha m(x)}u_\la (r\overline m+\la z) e^{-i\theta}-ihe^{\alpha m(x)}(r\overline m+\la z)=0\\
\tilde  g_2(z,r,\la):=(r^2-1)\overline m^2|\Omega|+\la^2\|z\|^2_{Y_{\mathbb{C}}}=0
\end{cases}
\end{equation}
Define
$\tilde G:(X_1)_{\mathbb C}\times \mathbb{R}^4\to
Y_{\mathbb C}\times \mathbb{R}$ by $\tilde G=(g_1,g_2)$, and we see that
$\tilde G(z,r,h,\theta,\la)=0$ is also uniquely solvable for $\la=0$.
\begin{lemma}\label{l252}
The following equation \begin{equation}\label{3.5G2}
\begin{cases}
\tilde G(z,r,h,\theta,0)=0\\
z\in (X_1)_{\mathbb{C}},\;h>0\;r\ge0,\; \theta\in[0,2\pi)\\
\end{cases}
\end{equation} has a unique solution $(z_{0},r_{0},h_{0},\theta_{0})$. Here
\begin{equation}\label{lastar2}
    r_{0}=1,\;\;\theta_{0}=\pi/2,\;\;h_{0}=\ds\f{\int_\Omega m(x)e^{\alpha m(x)}dx}{\int_\Omega e^{\alpha m(x)}dx},
\end{equation}
and $z_{0}\in(X_1)_{\mathbb C}$ is the unique solution of
\begin{equation}\label{lastari2}
\begin{split}
-\nabla\cdot[e^{\alpha m(x)}\nabla z]=&e^{\alpha m(x)}\left[m(x)-e^{\alpha m(x)}\overline m\right]\overline m
-e^{2\alpha m(x)}\overline m^2 e^{-i\theta_0}-ih_0e^{\alpha m(x)}\overline m.\\
\end{split}
\end{equation}

\end{lemma}
\begin{proof}
From Eq. \eqref{g12}, we see that $\tilde g_2(z,r,0)=0$ if and only if $r=r_{0}=1$. Note that
\begin{equation}
\begin{split}
\tilde g_1(z,r_{0},h,\theta,0)=&\nabla\cdot[e^{\alpha m(x)}\nabla z]+e^{\alpha m(x)}\left[m(x)-e^{\alpha m(x)}\overline m\right]\overline m\\
-&e^{2\alpha m(x)}\overline m^2 e^{-i\theta}-ihe^{\alpha m(x)}\overline m=0\\
\end{split}
\end{equation}
Then \begin{equation*}
\begin{cases}
\tilde g_1(z,r_0,h,\theta,0)=0\\
z\in (X_1)_{\mathbb{C}},\;h>0\;r\ge0,\; \theta\in[0,2\pi)\\
\end{cases}
\end{equation*} is solvable if and only if
\begin{equation}\label{cost2}
\begin{cases}
\overline m^2\int_\Omega e^{2\alpha m(x)}dx\sin\theta=h\overline m\int_\Omega e^{\alpha m(x)}dx
\\
\overline m^2\int_\Omega e^{2\alpha m(x)}dx\cos\theta=0\\
\end{cases}
\end{equation}
is solvable for a pair $(\theta,h)$ with $h>0$ and $\theta\in[0,2\pi)$.
Noticing that $$\overline m=\ds\f{\int_\Omega m(x)e^{\alpha m(x)}dx}{\int_\Omega e^{2\alpha m(x)} dx},$$  we have
\begin{equation}\label{hthe2}
\theta=\theta_{0}=\pi/2,\;\;h=h_{0}=\ds\f{\int_\Omega m(x)e^{\alpha m(x)}dx}{\int_\Omega e^{\alpha m(x)}dx}.
\end{equation}
Consequently, $\tilde g_1(z,r_{0},h_{0},\theta_{0},0)=0$ has a unique solution $z_{0}$, which satisfies Eq. \eqref{lastari2}.
\end{proof}
Then, we also have the following result on the solvability of  $\tilde G=0$ for $\la\in(0,\la^*]$.
\begin{theorem}\label{cha2}
There exist $\tilde \la^*>0$ and a continuously differentiable mapping
$\la\mapsto(z_\la,r_\la,h_\la,\theta_\la)$ from
$[0,\tilde \la^{*}]$ to $(X_1)_{\mathbb C}\times \mathbb{R}^3$ such that
$\tilde G(z_\la,r_\la,r_\la,\theta_\la,\la)=0$. Moreover, for $\la\in[0,\tilde \la^{*}]$,
\begin{equation}\label{3.6G2} \begin{cases}
\tilde G(z,r,h,\theta,\la)=0\\
z\in (X_1)_{\mathbb{C}},\;h,\;r\ge0, \;\theta\in[0,2\pi)\\
\end{cases}
\end{equation}
 has a unique solution $(z_\la,r_\la,h_\la,\theta_\la)$.
\end{theorem}
\begin{proof}
Let $\tilde T=(\tilde T_1,\tilde T_2):(X_1)_{\mathbb C}\times \mathbb{R}^3\mapsto
Y_{\mathbb C}\times \mathbb{R}$ be the Fr\'echet derivative of $\tilde G$ with respect to
$(z,r,h,\theta)$ at $(z_{0},r_{0},h_{0},\theta_{0},0)$.
An easy calculation yields
\begin{equation*}
\begin{split}
\tilde T_1(\chi,\kappa,\epsilon,\vartheta)=&
\nabla\cdot[e^{\alpha m(x)}\nabla z]+\kappa e^{\alpha m(x)}\left[m(x)-e^{\alpha m(x)}\overline m\right]\overline m
-\kappa e^{2\alpha m(x)}\overline m^2 e^{-i\theta_0}\\-&i\kappa h_0e^{\alpha m(x)}\overline m-i\epsilon e^{\alpha m(x)}\overline m+\vartheta e^{2\alpha m(x)}\overline m^2,\\
\tilde T_2(\kappa)=&2\kappa\overline m^2|\Omega|.
\end{split}
\end{equation*}
Then, we check that $\tilde T$ is
a bijection from $(X_1)_{\mathbb C}\times \mathbb{R}^3$ to $Y_{\mathbb
C}\times \mathbb{R}$, and we only need to verify that $T$ is an injective mapping. If $\tilde T_2(\kappa)=0$, then $\kappa=0$, and substituting $\kappa=0$ into $\tilde T_1(\chi,\kappa,\epsilon,\vartheta)=0$, we obtain $\vartheta=\epsilon=0$. Therefore, $T$ is an
an injection. It follows from the implicit function theorem that
there exist $\tilde \la^*>0$ and a continuously differentiable mapping
$\la\mapsto(z_\la,r_\la,h_\la,\theta_\la)$ from
$[0,\tilde\la^{*}]$ to $X_{\mathbb C}\times \mathbb{R}^3$ such that
$\tilde G(z_\la,r_\la,h_\la,\theta_\la,\la)=0$. By the arguments similar to Lemma \ref{cha}, the uniqueness can be proved, and here we omit the proof.
\end{proof}
Summarizing the above result, we have the following result.
\begin{theorem}\label{c252}
For each $\la\in(0,\tilde \la^*],$ the
following equation
\begin{equation*}
\begin{cases}
\Delta(\la,i\nu,\tau)\psi=0\\
\nu\geq0,\;\tau\ge0,\;\psi (\ne 0) \in X_{\mathbb C}\\
\end{cases}
\end{equation*}
has a solution $(\nu,\tau,\psi)$, if and
only if
\begin{equation}\label{par2}
\nu=\nu_\la=\la h_\la,\;\psi= c \psi_\la,\;
\tau=\tau_{n}=\frac{\theta_\la+2n\pi}{\nu_\la},\;\; n=0,1,2,\cdots,
\end{equation}
where $\psi_\la=r_\la\overline m+\la z_\la$,
$c$ is a nonzero constant, and
$z_\la,r_\la,h_\la,\theta_\la$ are defined as in Theorem
\ref{cha2}.
\end{theorem}
The simplicity of $i\nu$ and the transversality
condition can also be derived as in Lemma \ref{thm34}, Theorems \ref{thm34a} and \ref{thm35}, and we also omit the proof here. Therefore, for case II, we also derive the existence of Hopf bifurcation.
\begin{theorem}
Assume that $m(x)$ satisfies assumption $\mathbf{(A_3)}$. Then, for $\lambda\in(0,\tilde \lambda^*]$, where $0<\tilde \la^*\ll 1$, there exists a sequence $\{\tau_n\}_{n=0}^\infty$ such that
the positive steady state $u_\lambda$ of Eq. \eqref{delay2}
is locally asymptotically stable when $\tau\in[0,\tau_{0})$,
unstable when $\tau\in(\tau_{0},\infty)$, and system \eqref{delay2} occurs Hopf bifurcation at the positive steady state $u_{\la}$ when $\tau=\tau_n$,
$(n=0,1,2,\cdots)$.
\end{theorem}

\end{document}